\documentclass[12pt]{amsart}

\newcommand{\rk}{\operatorname{rk}}

\newtheorem{theorem}{Theorem}[section]
\newtheorem{proposition}[theorem]{Proposition}
\newtheorem{lemma}[theorem]{Lemma}
\newtheorem{corollary}[theorem]{Corollary}

\numberwithin{equation}{section}

\begin{document}

\baselineskip=16pt

\title[Moduli spaces of framed bundles]{Rationality and
Brauer group of a moduli space of framed bundles}

\author[I. Biswas]{Indranil Biswas}

\address{School of Mathematics, Tata Institute of Fundamental
Research, Homi Bhabha Road, Bombay 400005, India}

\email{indranil@math.tifr.res.in}

\author[T. L. G\'omez]{Tom\'as L. G\'omez}

\address{Instituto de Ciencias Matem\'aticas (CSIC-UAM-UC3M-UCM),
Nicol\'as Cabrera 15, Campus Cantoblanco UAM, 28049 Madrid, Spain}

\email{tomas.gomez@icmat.es}

\author[V. Mu\~{n}oz]{Vicente Mu\~{n}oz}

\address{Facultad de Ciencias Matem\'aticas, Universidad
Complutense de Madrid, Plaza Ciencias 3, 28040 Madrid, Spain}

\email{vicente.munoz@mat.ucm.es}



\date{}

\begin{abstract}
We prove that the moduli spaces of framed bundles over a
smooth projective curve are rational. We compute the Brauer group
of these moduli spaces to be zero under some assumption on the
stability parameter.
\end{abstract}

\maketitle

\section{Introduction}

Let $X$ be a compact connected Riemann surface of genus $g$,
with $g\, \geq\, 2$. A framed bundle on $X$ is a pair of the
form $(E\, ,\phi)$, where
$E$ is a vector bundle on $X$, and
$$
\phi\, :\, E_{x_0}\,\longrightarrow\, {\mathbb C}^r
$$
is a non--zero $\mathbb C$--linear homomorphism, where $r$
is the rank of $E$. The notion of a (semi)stable vector bundle
extends to that for a framed bundle. But the (semi)stability
condition depends on a parameter $\tau\,\in\, {\mathbb R}_{>0}$.
Fix a positive integer $r$, and also fix
a holomorphic line bundle $L$ over $X$. Also, fix a positive number
$\tau\, \in\, \mathbb R$. Let ${\mathcal M}^\tau_L(r)$ be the
moduli space of $\tau$--stable framed bundles of rank $r$
and determinant $L$.

In \cite{BGM}, we investigated the geometric structure of
the variety ${\mathcal M}^\tau_L(r)$. The following theorem
was proved in \cite{BGM}:

\textit{Assume that $\tau\, \in\, (0\, , \frac{1}{(r-1)! (r-1)})$.
Then the isomorphism class of the Riemann surface $X$ is
uniquely determined by the isomorphism class of the
variety ${\mathcal M}^\tau_L(r)$.}

Our aim here is to investigate the rationality properties
of the variety ${\mathcal M}^\tau_L(r)$. We prove the following
(see Theorem \ref{thm1} and Corollary \ref{cor1}):

\textit{The variety ${\mathcal M}^\tau_L(r)$ is rational.}

\textit{If $\tau\, \in\, (0\, , \frac{1}{(r-1)! (r-1)})$, then
$$
{\rm Br}({\mathcal M}^\tau_L(r))\,=\, 0\, ,
$$
where ${\rm Br}({\mathcal M}^\tau_L(r))$ is the Brauer group
of ${\mathcal M}^\tau_L(r)$.}

The rationality of ${\mathcal M}^\tau_L(r)$ is proved by showing
that ${\mathcal M}^\tau_L(r)$ is birational to the total space of
a vector bundle over the moduli space of stable vector bundles $E$
on $X$ together with a line in the fiber of $E$ over a fixed point.
The rationality of ${\mathcal M}^\tau_L(r)$ also follows 
from \cite{Ho-extra}, Example 6.9,
taking $D$ to be the point $x_0$.

The Brauer group of ${\mathcal M}^\tau_L(r)$ is computed by
considering the morphism to the usual moduli space that forgets
the framing.

\section{Rationality of moduli space}

Let $X$ be a compact connected Riemann surface of genus $g$,
with $g\, \geq\, 2$. Fix a holomorphic line bundle $L$ over
$X$, and take an integer $r>0$. Fix a point $x_0\,\in\, X$.
A framed coherent sheaf over $X$ is a pair of the form $(E\, ,\phi)$,
where $E$ is a coherent sheaf on $X$ of rank $r$, and
$$
\phi\, :\, E_{x_0}\,\longrightarrow\, {\mathbb C}^r
$$
is a non--zero $\mathbb C$--linear homomorphism.
Let $\tau>0$ be a real number. A framed coherent sheaf is called
$\tau$--\textit{stable} (respectively, $\tau$--\textit{semistable})
if for all proper subsheaves $E'\,\subset\, E$, we have
\begin{equation}
\label{eq:st}
{\deg E' - \epsilon(E',\phi)\tau}
< {\rk E'} \, \, \frac{\deg E - \tau}{\rk E}
\end{equation}
(respectively, ${\deg E' - \epsilon(E',\phi)\tau}
\leq {\rk E'}({\deg E - \tau})/{\rk E}$), where
$$
\epsilon(E',\phi)\,=\,\left\{
\begin{array}{lcl}
1 & \text{if} & \phi|_{E'_{x_0}}\,\neq\, 0, \\
0 & \text{if} & \phi|_{E'_{x_0}}\,=\, 0\, .
\end{array}
\right .
$$
A framed bundle is a framed coherent sheaf $(E,\varphi)$
such that $E$ is locally free.

We remark that the framed coherent sheaves considered here are
special cases of the objects
considered in \cite{HL}, and hence from \cite{HL}
we conclude that the moduli space ${\mathcal M}^\tau_L(r)$
of $\tau$--stable framed bundles of rank $r$ and determinant $L$
is a smooth quasi--projective variety.

Let $(E\, ,\phi)$ be a $\tau$--semistable framed
coherent sheaf.
We note that if $\tau<1$, then $E$ is necessarily
torsion--free, because a torsion subsheaf of $E$ will
contradict $\tau$--semistability,
hence in this case $E$ is locally free.
But if $\tau$ is large, then $E$ can have torsion.
In particular, the natural compactification
of ${\mathcal M}^\tau_L(r)$ using $\tau$--semistable framed
coherent sheaves could have points which are not framed bundles.

\begin{lemma}\label{lem1}
There is a dense Zariski open subset
\begin{equation}\label{e1}
{\mathcal M}^\tau_L(r)^0\, \subset\, {\mathcal M}^\tau_L(r)
\end{equation}
corresponding to pairs
$(E\, ,\phi)$ such that $E$ is a stable vector bundle
of rank $r$, and $\phi$ is an isomorphism.

The moduli space ${\mathcal M}^\tau_L(r)$ is irreducible.
\end{lemma}

\begin{proof}
{}From the openness of the stability condition
it follows immediately that the locus of framed bundles
$(E\, ,\phi)$ such that $E$ is not stable is a closed subset
of the moduli space ${\mathcal M}^\tau_L(r)$
(see \cite[p. 635, Theorem 2.8(B)]{Ma}
for the openness of the stability condition). It is
easy to check that the locus of framed bundles
$(E\, ,\phi)$ such that $\phi$ is not an isomorphism
is a closed subset of ${\mathcal M}^\tau_L(r)$. Therefore,
${\mathcal M}^\tau_L(r)^0$ is a Zariski open subset of
${\mathcal M}^\tau_L(r)$.

We will now show that this open subset ${\mathcal M}^\tau_L(r)^0$
is dense.
Let $(E,\varphi)$ be a $\tau$--stable framed bundle.
The moduli stack of stable vector bundles is dense
in the moduli stack of coherent sheaves,
and both stacks are irreducible (see,
for instance, \cite[Appendix]{Ho}). Therefore we
can construct a family $\{E_t\}_{t\in T}$ of vector
bundles parametrized by an irreducible
smooth curve $T$ with a base point
$0\, \in\, T$ such that the following two conditions hold:
\begin{enumerate}
\item $E_0\cong E$, and

\item the vector bundle $E_t$ is stable for all
$t\,\not=\, 0$.
\end{enumerate}
Shrinking $T$ if necessary (by taking a nonempty Zariski open
subset of $T$), we get a family
of frames $\{\phi_t\}_{t\in T}$ such that $\phi_0$ is the
given frame $\phi$, and $\phi_t\,:\,E_{t,x_0}\,\longrightarrow\,
{\mathbb C}^r$
is an isomorphism for all $t\,\not=\, 0$. Since $E_t$ is stable,
and $\phi_t$ is an isomorphism, it is easy to check that
$(E_t,\phi_t)$ is $\tau$--stable. Therefore, ${\mathcal
M}^\tau_L(r)^0$ is dense in ${\mathcal M}^\tau_L(r)$.

To prove that ${\mathcal M}^\tau_L(r)$ is irreducible, first note
that ${\mathcal M}^\tau_L(r)^0$ is irreducible because the moduli
stack of stable vector bundles of fixed rank and determinant is
irreducible. Since ${\mathcal M}^\tau_L(r)^0\,\subset\,
{\mathcal M}^\tau_L(r)$
is dense, it follows that ${\mathcal M}^\tau_L(r)$ is irreducible.
\end{proof}

Let ${\mathcal N}_P$ be the moduli space of pairs of the form
$(E\, ,\ell)$, where $E$ is a stable vector bundle on $X$
of rank $r$ with determinant $L$, and $\ell\, \subset\,
E_{x_0}$ is a line. Consider ${\mathcal M}^\tau_L(r)^0$ defined in
\eqref{e1}. Let
\begin{equation}\label{e2}
\beta\, :\, {\mathcal M}^\tau_L(r)^0\, \longrightarrow\, {\mathcal N}_P
\end{equation}
be the morphism defined by $(E\, ,\phi)\, \longmapsto\,
(E\, ,\phi^{-1}({\mathbb C}\cdot e_1))$, where the standard basis of
${\mathbb C}^r$ is denoted by $\{e_1\, ,\ldots\, ,e_r\}$.

\begin{proposition}\label{prop1}
The variety ${\mathcal M}^\tau_L(r)^0$ is birational to the total
space of a vector bundle over ${\mathcal N}_P$.
\end{proposition}

\begin{proof}
We will first construct a tautological vector bundle over
${\mathcal N}_P$. Let ${\mathcal N}_L(r)$ be the moduli space
of stable vector bundles on $X$ of rank $r$ and determinant $L$.
Consider the projection
\begin{equation}\label{f}
f\, :\, {\mathcal N}_P\, \longrightarrow\, {\mathcal N}_L(r)
\end{equation}
defined by $(E\, ,\ell)\,\longrightarrow\, E$. Let
$P_{\text{PGL}}\, \longrightarrow\, {\mathcal N}_L(r)$ be the
principal $\text{PGL}(r, {\mathbb C})$--bundle corresponding to $f$;
the fiber of $P_{\text{PGL}}$ over any $E\, \in\, {\mathcal N}_L(r)$
is the space of all linear isomorphisms from $P({\mathbb C}^r)$
(the space of lines in ${\mathbb C}^r$)
to $P(E_{x_0})$ (the space of lines in $E_{x_0}$); since the
automorphism
group of $E$ is the nonzero scalar multiplications (recall that
$E$ is stable), the projective space $P(E_{x_0})$ is canonically
defined
by the point $E$ of ${\mathcal N}_L(r)$. Let
$$
Q\, \subset\, \text{PGL}(r, {\mathbb C})
$$
be the maximal parabolic subgroup that fixes the point of
$P({\mathbb C}^r)$ representing the line ${\mathbb C}\cdot e_1$.
The principal $\text{PGL}(r, {\mathbb C})$--bundle
$$
f^*P_{\text{PGL}}\, \longrightarrow\, {\mathcal N}_P
$$
has a tautological reduction of structure group
$$
\widetilde{E}_Q\, \subset\, f^*P_{\text{PGL}}
$$
to the parabolic subgroup $Q$; the fiber of $\widetilde{E}_Q$ over
any point $(E\, ,\ell)\, \in\, {\mathcal N}_P$ is the space of all
linear isomorphisms
$$
\rho\, :\, P({\mathbb C}^r)\,\longrightarrow\, P(E_{x_0})
$$ such that $\rho({\mathbb C}\cdot e_1) \,=\, \ell$.
The standard action of $\text{GL}(r, {\mathbb C})$
on ${\mathbb C}^r$ defines an action of $Q$ on $({\mathbb
C}\cdot e_1)^*\bigotimes_{\mathbb C} {\mathbb C}^r$. Let
\begin{equation}\label{e3}
W\, :=\, f^*P_{\text{PGL}}(({\mathbb C}\cdot
e_1)^*\otimes {\mathbb C}^r)\, \longrightarrow\, {\mathcal N}_P
\end{equation}
be the vector bundle over ${\mathcal N}_P$ associated to the
principal $\text{PGL}(r, {\mathbb C})$--bundle $f^*P_{\text{PGL}}$
for the above $\text{PGL}(r, {\mathbb C})$--module $({\mathbb
C}\cdot e_1)^*\bigotimes_{\mathbb C} {\mathbb C}^r$. The action
of $Q$ on
$({\mathbb C}\cdot e_1)^*\bigotimes_{\mathbb C} {\mathbb C}^r$ fixes
$$
\text{Id}_{{\mathbb C}\cdot e_1}\, \in\,
({\mathbb C}\cdot e_1)^*\otimes_{\mathbb C} {\mathbb C}^r
\,=\, \text{Hom}({\mathbb C}\cdot e_1\, , {\mathbb C}^r)\, .
$$
Therefore, the element $\text{Id}_{{\mathbb C}\cdot e_1}$ defines a
nonzero section
\begin{equation}\label{e4}
\sigma\, \in\, H^0({\mathcal N}_P,\, W)\, ,
\end{equation}
where $W$ is the vector bundle in \eqref{e3}. Note that the fiber of
$W$ over $(E,\ell)$ is $\ell^* \otimes E_{x_0}$, and the
evaluation of $\sigma$ at $(E,\ell)$ is $\text{Id}_\ell$.

The projective bundle $P(W)\, \longrightarrow\, {\mathcal N}_P$
parametrizing lines in $W$ is identified with the pullback
$f^*{\mathcal N}_P$ of the projective
bundle ${\mathcal N}_P$ to the total
space of ${\mathcal N}_P$, where $f$ is constructed in \eqref{f}.
The tautological section
${\mathcal N}_P\, \longrightarrow\, f^*{\mathcal N}_P$ of the
projection $f^*{\mathcal N}_P\, \longrightarrow\,{\mathcal N}_P$
coincides with the section given by $\sigma$ in \eqref{e4}.

Let $U\, \subset\, {\mathcal N}_P$ be some nonempty Zariski open
subset such that there exists
$$
V\, \subset\, W\vert_{U}\, ,
$$
a direct summand of the line subbundle of $W\vert_{U}$ generated
by $\sigma$. Consider the vector bundle
$$
{\mathcal W}\, :=\, V^*\otimes_{\mathbb C} {\mathbb C}^r
\, \longrightarrow\, U\, .
$$
The total space of $\mathcal W$ will also be denoted by
$\mathcal W$. Consider the map $\beta$ defined in \eqref{e2}. Let
$$
\gamma\,:\, {\mathcal M}^\tau_L(r)^0 \, \supset\,
\beta^{-1}(U)\, \longrightarrow\, \mathcal W
$$
be the morphism that sends any $y\,:=\, (E\, ,\phi)\,\in\,
\beta^{-1}(U)$ to the homomorphism
$$
V_{\beta(y)}\, \longrightarrow\, {\mathbb C}^r
$$
defined by $v\, \longmapsto\, \phi(v)/\lambda$, where
$\lambda\, \in\, {\mathbb C}^*-\{0\}$ 
satisfies the identity
$\phi(\sigma(\beta(y)))\,=\, \lambda\cdot e_1$.
The morphism $\gamma$ is clearly birational.
\end{proof}

\begin{theorem}\label{thm1}
The moduli space ${\mathcal M}^\tau_L(r)$ is rational.
\end{theorem}

\begin{proof}
Since any vector
bundle is Zariski locally trivial, the total space
of a vector bundle of rank $n$ over ${\mathcal N}_P$ is
birational to ${\mathcal N}_P\times {\mathbb A}^n$.
Therefore, from Proposition \ref{prop1} we conclude that
${\mathcal M}^\tau_L(r)^0$ is birational to ${\mathcal N}_P\times
{\mathbb A}^n$, where $n\,=\, \dim {\mathcal M}^\tau_L(r)^0
-\dim {\mathcal N}_P$.

The variety ${\mathcal N}_P$ is known to be rational
\cite[p.\ 472, Theorem 6.2]{BY}. Hence
${\mathcal N}_P\times {\mathbb A}^n$ is rational, implying
that ${\mathcal M}^\tau_L(r)^0$ is rational. Now from Lemma
\ref{lem1} we infer that ${\mathcal M}^\tau_L(r)$ is rational.
\end{proof}

\section{Brauer group of moduli of framed bundles}

We quickly recall the definition of Brauer group of a variety $Z$.
Using the natural isomorphism ${\mathbb C}^r\otimes
{\mathbb C}^{r'}\,\stackrel{\sim}{\longrightarrow}\,{\mathbb C}^{rr'}$,
we have a homomorphism
$\text{PGL}(r,{\mathbb C})\times \text{PGL}(r',{\mathbb C})
\,\longrightarrow\, \text{PGL}(rr',{\mathbb C})$. So a
principal $\text{PGL}(r,{\mathbb C})$--bundle $\mathbb P$ and
a principal $\text{PGL}(r',{\mathbb C})$--bundle
${\mathbb P}'$ on $Z$ together produce
a principal $\text{PGL}(rr',{\mathbb C})$--bundle
on $Z$, which we will denote by ${\mathbb P}\otimes {\mathbb P}'$.
The two principal
bundles $\mathbb P$ and ${\mathbb P}'$
are called \textit{equivalent} if there are vector bundles
$V$ and $V'$ on $Z$ such that the principal bundle
${\mathbb P}\otimes {\mathbb P}(V)$ is isomorphic to
${\mathbb P}'\otimes {\mathbb P}(V')$. The equivalence classes
form a group which is called the \textit{Brauer group} of $Z$.
The addition operation is defined
by the tensor product, and the inverse is defined to be the dual
projective bundle. The Brauer group of $Z$ will be
denoted by $\text{Br}(Z)$.

As before, fix $r$ and $L$. Define
$$
\tau(r) \, :=\, \frac{1}{(r-1)! (r-1)}\, .
$$
Henceforth, we assume that
$$
\tau\, \in\, (0\, , \tau(r))\, ,
$$
where $\tau$ is the parameter in the definition of a
(semi)stable framed bundle. As before, let ${\mathcal M}^\tau_L(r)$
be the moduli space of $\tau$--stable framed bundles of rank
$r$ and determinant $L$.

Let $\overline{\mathcal N}_L(r)$ be the moduli space of
semistable vector bundles on $X$ of rank $r$ and determinant
$L$. As in the previous section, the moduli space of stable
vector bundles on $X$ of rank
$r$ and determinant $L$ will be denoted by ${\mathcal N}_L(r)$.

If $E$ is a stable vector bundle of rank $r$ and determinant
$L$, then for any nonzero homomorphism
$$
\phi\, :\, E_{x_0}\, \longrightarrow\, {\mathbb C}^r\, ,
$$
the framed bundle $(E\, ,\phi)$ is $\tau$--stable (see
\cite[Lemma 1.2(ii)]{BGM}). Also, if $(E\, ,\phi)$ is any
$\tau$--stable framed bundle, then $E$ is semistable
\cite[Lemma 1.2(i)]{BGM}. Therefore, we have a morphism
\begin{equation}\label{delta}
\delta\, :\, {\mathcal M}^\tau_L(r)\, \longrightarrow\,
\overline{\mathcal N}_L(r)
\end{equation}
defined by $(E\, ,\phi)\, \longrightarrow\,E$. Define
\begin{equation}\label{d2}
{\mathcal M}^\tau_L(r)'\, :=\, \delta^{-1}({\mathcal N}_L(r))\,
\subset\, {\mathcal M}^\tau_L(r)\, ,
\end{equation}
where $\delta$ is the morphism in \eqref{delta}. From the
openness of the stability condition (mentioned in the
proof of Lemma \ref{lem1}) it follows that ${\mathcal M}^\tau_L(r)'$
is a Zariski open subset of ${\mathcal M}^\tau_L(r)$.

\begin{lemma}\label{lem2}
The Brauer group of the variety ${\mathcal M}^\tau_L(r)'$ vanishes.
\end{lemma}

\begin{proof}
We noted above that $(E\, ,\phi)$ is $\tau$--stable if
$E$ is stable. Therefore, the morphism
$$
\delta_1\,:=\, \delta\vert_{{\mathcal M}^\tau_L(r)'}\, :\,
{\mathcal M}^\tau_L(r)'\, \longrightarrow\,{\mathcal N}_L(r)
$$
defines a projective bundle over ${\mathcal N}_L(r)$,
where $\delta$ is constructed in \eqref{delta}; for notational
convenience, this projective bundle ${\mathcal M}^\tau_L(r)'$ will
be denoted by ${\mathcal P}$. The homomorphism
$$
\delta^*_1\, :\, \text{Br}({\mathcal N}_L(r))\,\longrightarrow
\, \text{Br}({\mathcal P})
$$
is surjective, and the kernel of $\delta^*_1$ is generated
by the Brauer class
$$
\text{cl}({\mathcal P})\, \in\, \text{Br}({\mathcal
N}_L(r))
$$
of the projective bundle ${\mathcal P}$
(see \cite[p.\ 193]{Ga}). In other words, we have an exact
sequence
\begin{equation}\label{es}
{\mathbb Z}\cdot \text{cl}({\mathcal P})\,\longrightarrow
\, \text{Br}({\mathcal N}_L(r))\,
\stackrel{\delta^*_1}{\longrightarrow}\,\text{Br}
({\mathcal M}^\tau_L(r)')\, \longrightarrow\, 0\, .
\end{equation}

Let
$$
{\mathbb P}\,:=\, {\mathcal N}_L(r)\times P({\mathbb C}^r)\,
\longrightarrow\, {\mathcal N}_L(r)
$$
be the trivial projective bundle over ${\mathcal N}_L(r)$.
Consider the projective bundle
$$
f\, :\, {\mathcal N}_P\, \longrightarrow\, {\mathcal N}_L(r)
$$
in \eqref{f}. Let
$$
({\mathcal N}_P)^*\, \longrightarrow\, {\mathcal N}_L(r)
$$
be the dual projective bundle; so the fiber of
$({\mathcal N}_P)^*$ over any point $z\, \in\,
{\mathcal N}_L(r)$ is the space of all hyperplanes in the
fiber of ${\mathcal N}_P$ over $z$. It is easy to see that
\begin{equation}\label{f2}
{\mathcal P}\,=\, ({\mathcal N}_P)^*\otimes {\mathbb P}
\end{equation}
(the tensor product of two projective bundles was defined
at the beginning of this section).

Since ${\mathbb P}$ is a trivial projective bundle, from
\eqref{f2} it follows that
$$
\text{cl}({\mathcal P})\, =\, \text{cl}(({\mathcal N}_P)^*)
\,=\, -\text{cl}({\mathcal N}_P)
\, \in \, \text{Br}({\mathcal N}_L(r))\, .
$$
But the Brauer group $\text{Br}({\mathcal N}_L(r))$ is generated
by $\text{cl}({\mathcal N}_P)$ \cite[Proposition 1.2(a)]{BBGN}.
Hence $\text{cl}({\mathcal P})$ generates
$\text{Br}({\mathcal N}_L(r))$. Now from \eqref{es} we conclude
that $\text{Br}({\mathcal M}^\tau_L(r)')\,=\, 0$.
\end{proof}

\begin{corollary}\label{cor1}
The Brauer group of the moduli space ${\mathcal M}^\tau_L(r)$
vanishes.
\end{corollary}

\begin{proof}
Since ${\mathcal M}^\tau_L(r)'$ is a nonempty Zariski open
subset of ${\mathcal M}^\tau_L(r)$, the homomorphism
$$
\text{Br}({\mathcal M}^\tau_L(r))\, \longrightarrow\,
\text{Br}({\mathcal M}^\tau_L(r)')
$$
induced by the inclusion ${\mathcal M}^\tau_L(r)'\,
\hookrightarrow\, {\mathcal M}^\tau_L(r)$ is injective. 
Therefore, from Lemma \ref{lem2} it follows that
$\text{Br}({\mathcal M}^\tau_L(r))\,=\, 0$.
\end{proof}


\end{document}